\documentclass[11pt]{amsart}

\usepackage{
amsfonts,
amsmath,
latexsym,
amssymb,
enumerate,
mathrsfs,
color,
}

\usepackage[matrix,arrow]{xy}

\newcommand{\labbel}[1]{\label{#1} [[{\bf #1}]]}  
\renewcommand{\labbel}{\label}

\newtheorem{theorem}{Theorem}[section]
\newtheorem{lemma}[theorem]{Lemma}

\newtheorem{claim}[theorem]{Claim} 
\newtheorem*{claim*}{Claim}

\newtheorem*{theorem*}{Theorem}
\newtheorem*{proposition*}{Proposition}
\newtheorem*{corollary*}{Corollary}
\newtheorem*{lemma*}{Lemma}
\newtheorem*{scholion*}{Scholion}

\theoremstyle{definition}
\newtheorem{definition}[theorem]{Definition}

\theoremstyle{remark}
\newtheorem{remark}[theorem]{Remark}

\newtheorem*{remark*}{Remark}
\newtheorem*{remarks*}{Remarks}

\newtheorem*{observation*}{Observation}

\numberwithin{equation}{section}

\begin{document}

\title{Universal specialization semilattices}

\author{Paolo Lipparini} 

\email{lipparin@axp.mat.uniroma2.it}

\urladdr{http://www.mat.uniroma2.it/~lipparin}

\address{Dipartimento di Matematica\\Viale della  Ricerca
 Scientifica Non Additiva \\Universit\`a di Roma ``Tor Vergata'' 
\\I-00133 ROME ITALY}



\date{\today}

\thanks{Work performed under the auspices of G.N.S.A.G.A. 
Work partially supported by PRIN 2012 ``Logica, Modelli e Insiemi''.
The author acknowledges the MIUR Department Project awarded to the
Department of Mathematics, University of Rome Tor Vergata, CUP
E83C18000100006.}

\begin{abstract}
A specialization semilattice is a structure which
can be embedded into 
$(\mathcal P(X), \cup, \sqsubseteq )$,
where 
 $X$ is a topological space,
$ x \sqsubseteq y$ means $x \subseteq Ky$,
   for $x,y \subseteq  X$, and $K$ is closure in $X$.
Specialization semilattices and posets
appear as auxiliary structures in many 
disparate scientific fields, even unrelated to topology.

In general, closure is not expressible in a specialization semilattice.
On the other hand, we show that every 
 specialization semilattice can be canonically embedded into a
``principal'' specialization semilattice 
in which closure can be actually defined.
 \end{abstract}

\maketitle  

\section{Introduction} \labbel{intro} 

Due to the importance of the notion of \emph{closure}
 in mathematics, both in the topological sense and in
 the sense of \emph{hull}, \emph{generated by}\dots, 
it is interesting to study \emph{closure
spaces}  and related structures \cite{E}.
For example, a \emph{specialization semilattice} \cite{mtt} 
is a join-semilattice together with a  binary relation
$ \sqsubseteq $ such that 
\begin{align}
\labbel{s1}    \tag{S1}
 & a \leq b  \Rightarrow  a \sqsubseteq b, 
\\
\labbel{s2}    \tag{S2}
& a \sqsubseteq b \ \&\   b \sqsubseteq c \Rightarrow 
 a \sqsubseteq c, 
\\
\labbel{s3}    \tag{S3}
 &a \sqsubseteq b  \ \&\  a_1 \sqsubseteq b
 \Rightarrow 
 a \vee a_1 \sqsubseteq b, 
\end{align}
for all $a,b,c, a_1$. Thus $ \sqsubseteq $ is a preorder
coarser than the order $\leq$ induced by $\vee$, 
and furthermore the compatibility condition \eqref{s3} is satisfied. 

The main motivating example
for the notion is the following.
If $X$ is a topological space  with closure $K$,  
then   $(\mathcal P(X), \cup, \sqsubseteq )$ is a 
 specialization semilattice,
where $ \sqsubseteq $ is defined by
\begin{equation*}  
a \sqsubseteq b \quad \text{ if } \quad  a \subseteq Kb, 
\end{equation*}     
for all $a, b \subseteq X$.

Continuous functions between topological spaces 
are exactly those functions for which 
the corresponding direct image function preserves 
the specialization structure.
We have proved in  \cite[Theorem 4.10]{mtt}
that every specialization semilattice can be embedded into
the specialization semilattice associated  as above to some topological
space. In a sense, specialization semilattices represent
those point-free topological aspects which are preserved under taking continuous
direct images---at least, in the sense of first-order
model theory.

 However, there are many more examples of specialization semilattices.
A classical example is inclusion modulo finite \cite{Bl}.
If $X$ is an infinite set and we let $ x \sqsubseteq y$
if $x \setminus y$ is finite, for $x,y \subseteq X$,
then $(\mathcal P(X), \cup, \sqsubseteq )$
is a specialization semilattice. 
In this example no recognizable
notion of  ``closure'' is present.
The above example implicitly 
involves a quotient structure, as we are
going to explain.

If $\varphi: \mathbf S \to \mathbf T $  is a semilattice homomorphism
 and we set $ a \sqsubseteq_ \varphi    b$ in $S$ 
when $\varphi(a) \leq _{_\mathbf T}  \varphi (b)$, then 
$\mathbf S$ 
acquires the structure of a specialization 
semilattice.  
This quotient construction appears 
in many disparate settings, in various fields 
and with a wide range of applications.
The case of inclusion modulo finite  corresponds to considering
$\sqsubseteq_ \varphi $ as introduced above
when $\varphi$  is   the projection
from the standard Boolean algebra on  $\mathcal P(X)$
to its quotient modulo the ideal of finite sets.

Various similar structures appear, e.~g.,  in theoretical studies 
related to measure theory \cite{L},
algebraic logic \cite[Subsection 3.1]{GT} and even theoretical physics \cite{KP}.
Henceforth specialization semilattices are an interesting object
of study.
See \cite{mtt,mttlib} for further examples and details.

We now discuss the connection between specialization and closure
in more detail.   
A \emph{closure operation} on a semilattice
is a unary operation $K$ which is \emph{extensive},
 \emph{idempotent} and \emph{isotone},
namely, $x \leq Kx$,  $KKx=Kx$ 
hold for all elements $x$ of the poset,
and furthermore 
$x \leq y$ implies $Kx \leq Ky$.   
A \emph{closure semilattice} is a join-semilattice endowed with a closure operation.
Notice that the definition of a closure operation can be 
already given in the
setting of posets, where, as usual, 
\emph{poset} is a shorthand
for \emph{partially ordered set}.
On the other hand, an \emph{additive} closure operation
on a join-semilattice is required to satisfy also 
$K(x \vee y) =Kx \vee Ky$. This notion makes sense
only in the setting of structures which have at least a semilattice
operation.
Unless otherwise explicitly mentioned, here we shall
\emph{not} assume that closure operations are additive.
See \cite[Section 3]{E} for  a detailed study of closure posets
and semilattices,
with many applications.

Similar to the case of topologies, a closure operation on a poset induces
a \emph{specialization $ \sqsubseteq $} as follows:
$ x \sqsubseteq y$ if $x \leq Ky$.
Thus if $(P, \vee, K)$ is a closure semilattice,
then  $(P, \leq, \sqsubseteq )$ is a specialization semilattice. 
From the latter structure we can retrieve the original closure $K$: indeed,
$Kx$ is the $\leq$-largest element $y$ such that $ y \sqsubseteq x$.  
However, as shown by the example of inclusion modulo finite, 
such a largest element does not necessarily exist
  in an arbitrary specialization semilattice.
If the specialization semilattice $\mathbf S$  is such that, for every $x \in S $,
there exists the  largest element $y \in S$ such that $ y \sqsubseteq x$,
then $\mathbf S$ is said to be \emph{principal}. It can be shown
that there is a bijective correspondence between closure semilattices and
principal specialization semilattices. See Remark \ref{psc} below.  

In this note we show that every specialization semilattice 
has a canonical ``free'' extension into a principal specialization semilattice.
In particular, every specialization semilattice can be embedded into
the specialization reduct of some closure semilattice.
The somewhat simpler  case of  additive closure semilattices
 has been treated in 
\cite{mttlib}.

\section{Preliminaries} \labbel{prel} 

The definitions of a specialization and of  a
closure semilattice have been given in the introduction.
\emph{Homomorphisms} 
and \emph{embeddings} between 
such structures are always intended
in the classical model-\hspace{0 pt}theoretical sense \cite{H}.
In detail, a homomorphism
of specialization semilattices
is a semilattice homomorphism $\eta$
such that $ a \sqsubseteq b$ implies $\eta(a) \sqsubseteq  \eta(b)$
and an embedding is an injective homomorphism such that 
also the converse holds.
A homomorphism of closure semilattices 
is a semilattice homomorphism satisfying
$ \eta (Ka)= K \eta (a)$. 

If $\mathbf S$ is a specialization semilattice,
$a \in S$ and the set
$S_a=\{ b \in S \mid  b \sqsubseteq a \}$
has a $\leq$-maximum, such a maximum shall be denoted
by $Ka$ and shall be called the \emph{closure} of $a$.   
We require $Ka$ to be the maximum
of $S_a$, not just a supremum, namely, we 
actually require $Ka \sqsubseteq a$.  An
element $b$ of a specialization semilattice is called a
\emph{closure} if $b=Ka$, for some $a \in S$.  
As mentioned in the introduction, in general, $Ka$ need not exist
in an arbitrary specialization semilattice.
If $Ka$ exists for every  
$a \in S$, then $\mathbf S$ shall be called a
\emph{principal} specialization semilattice.
Homomorphisms of specialization semilattices
do not necessarily preserve closures.
If $\eta$ is a homomorphism between two \emph{principal}
specialization semilattices, then $\eta$ is a \emph{$K$-homomorphism} 
if  $ \eta (Ka)= K \eta (a)$, for every $a$. Thus $K$-homomorphisms
are actually homomorphisms for the associated
closure semilattices.

\begin{remark} \labbel{psc}    
(a) There is a bijective correspondence between 
principal specialization semilattices and closure semilattices.
As we mentioned, if $\mathbf  C$ is a closure semilattice, then
setting $ a \sqsubseteq b$
if $a \leq Kb$ makes $\mathbf  C$ 
a specialization semilattice.
Then $K$  turns out to 
be closure also
   in the  sense of specialization semilattices.
Conversely, if $\mathbf S$ is a \emph{principal}
specialization semilattice, then closure $K$, as defined above,
is actually a closure operation, hence $\mathbf S$ can be endowed with
the structure of a closure semilattice. 
See \cite[Section 3.1]{E}, in particular,
\cite[Proposition 3.9]{E} for details.     
 
On the other hand, the notions of homomorphism
are distinct. A homomorphism of closure semilattices
is a $K$-homomorphism of the corresponding specialization semilattices,
in particular,  a homomorphism of specialization semilattices.
The converse does not necessarily hold,
as mentioned above.

(b) In any closure poset, and hence in any
principal specialization semilattice, the following
identity holds.
 \begin{equation}\labbel{top}   
K(a_1 \vee \dots \vee a_r \vee Kb_1 \vee \dots \vee Kb_s)
= K(a_1  \vee \dots \vee a_r \vee b_1 \vee \dots \vee b_s)
  \end{equation}   
 for all $a_1, \dots, b_s$. This is a classical result,
at least in the case of closure spaces.
See \cite[Remark 2.1(b)]{mttlib} for the case of two summands and
\cite{sapimpap} for a proof in the general case.

(c) If $a$ and $b$ are elements of  some specialization semilattice 
and both $Ka$ and  $Kb$ exist, then
$Ka \leq Kb$ if and only if $ a \sqsubseteq b$.
See \cite[Remark 2.1(c)]{mttlib} for a proof.
In particular, in a principal specialization semilattice,
$Ka=Kb$ if and only if both $a \sqsubseteq b$ and  
$b \sqsubseteq a$. 

(d) From \eqref{s1} and reflexivity of $\leq$ it follows that 
\begin{align} 
\labbel{s4}    \tag{S4} 
 & a  \sqsubseteq a,
 \end{align}  
for every $a $ in $ S$. 

(e) \cite[Lemma 4.1]{mtt} 
If $\mathbf S$ is a specialization poset,
$b \in S$ and  $Kb$ exists, then, for every $a \in S$,
the following conditions are equivalent.
 \begin{enumerate}[(i)]
   \item
$ a \sqsubseteq b $;
\item
$ a \leq Kb$;
\item
$ a \sqsubseteq Kb $.
   \end{enumerate} 
The equivalence of (ii) and (iii) shows that if $Kb$
exists, then $KKb$ exists, as well, and in fact
$KKb=Kb$.   
\end{remark}

\section{Free principal extensions} \labbel{secimp} 

The existence of ``universal''
extensions and morphisms like
the ones we are going to construct
 follows from abstract categorical arguments;
see, e.~g., \cite[Lemma 4.1]{mttlib}.
On the other hand, an explicit description of
universal objects is sometimes very hard to find.
In the terminology from   \cite[Section 4]{mttlib}
here we are going to construct the universal 
object corresponding to case (C2).

\begin{remark} \labbel{rough}
Roughly, we work on the ``free''
semilattice extension  $ \widetilde { \mathbf   S} $
 of $\mathbf S$ obtained by adding
a  set of new elements
$\{ \, Ka \mid a \in S \,\}$,
subject to the further conditions
\begin{equation*}
\text{$a \leq Ka \leq  Kb$, whenever $a \sqsubseteq_{_{ \mathbf   S}} b$ 
 \quad (compare Remark \ref{psc}(c))}
\end{equation*} 
(here and below, the ``old'' relations and operations 
will be denoted with the subscript of the parent structure, while the
``new'' ones
will be unsubscripted).
Of course, we want to define $ \sqsubseteq $ 
 in $ \widetilde { \mathbf   S} $
in such a way that the new element $Ka$
is actually the closure of $a$. 
By transitivity of $\leq$, and if we want $K$
obey its defining property in a principal specialization semilattice,
 then  we necessarily have
  \begin{enumerate}[$\bullet$]
    \item  
$a \leq c \vee Kd_1 \vee \dots \vee Kd_h$, whenever 
there are elements $d^*_1 \sqsubseteq_{_{ \mathbf   S}} d_1$, \dots, 
$d^*_h \sqsubseteq_{_{ \mathbf   S}} d_h$ in $\mathbf S$ such that 
$a \leq_{_{ \mathbf   S}} c \vee_{_{ \mathbf   S}}
 d^*_1 \vee_{_{ \mathbf   S}} \dots \vee_{_{ \mathbf   S}} d^*_h$, and 
\item
$Kb \leq c \vee Kd_1 \vee \dots \vee Kd_h$, whenever 
there is $j \leq k$ such that   $b \sqsubseteq_{_{ \mathbf   S}} d_i$.
  \end{enumerate}

The above considerations justify  clauses
(a) - (b) in Definition \ref{und2} below.
Compare Remark \ref{*} below.

Notice that the  construction hinted above adds
a new closure of $a$ in the extension, even when $a$ has already
a closure in $\mathbf S$.   See Section \ref{pec} 
for a more involved construction in which existing closures
(or, more generally, some specified set of closures) are preserved.
\end{remark}   
 
\begin{definition} \labbel{sum}
If $S$
is any set,  
let $\mathbf S^{{<} \omega }$
be the semilattice 
of the finite subsets of $S$,
with the operation of union.
 \end{definition}   

\begin{definition} \labbel{und2}   
Suppose that 
$\mathbf S=(S, \vee_{_{ \mathbf   S}}, \sqsubseteq _{_{ \mathbf   S}})$
 is a specialization semilattice.
On the  product
$ S \times  S^{{<} \omega }$ 
define the following relations
  \begin{enumerate}[(a)] 
\item
 $(a, \{b_1, b_2, \dots, b_h\}) \precsim 
(c,\{d_1, d_2, \dots, d_k\})$ if
   \begin{enumerate}[({a}1)]  
  \item 
there is a way of choosing in $S$,
for every $j \leq k$,  
some $d_j^* \sqsubseteq_{_{ \mathbf   S}} d_j$
in such a way  that
 $a \leq_{_{ \mathbf   S}} c
 \vee_{_{ \mathbf   S}} d_1^* 
 \vee_{_{ \mathbf   S}} d_2^* 
\vee_{_{ \mathbf   S}} \dots 
\vee_{_{ \mathbf   S}} d_k^*$, and
   \item 
for every $i \leq h$, 
there is $j \leq k$ such that   
$b_i \sqsubseteq_{_{ \mathbf   S}} d_j$.
 \end{enumerate}
Clause (a1) is intended in the sense that we must have  $a \leq_{_{ \mathbf   S}} c $, 
when $k=0$, that is, when $\{d_1, d_2, \dots, d_k\}$ is empty.
\item
 $(a, \{b_1, b_2, \dots, b_h\}) \sim 
(c,\{d_1, d_2, \dots, d_k\})$ if both

 $(a, \{b_1, b_2, \dots, b_h\}) \precsim 
(c,\{d_1, d_2, \dots, d_k\})$ and

 $(c,\{d_1, d_2, \dots, d_k\}) \precsim 
(a, \{b_1, b_2, \dots, b_h\}) $.
  \end{enumerate}

We shall prove in Lemma \ref{corre2}
that $\sim$ is an equivalence relation.   

Let 
$ \widetilde {  S} $
be the quotient of  
$ S \times S^{{<} \omega }$ 
under the equivalence relation $\sim$.

Define $K: \widetilde {   S} \to \widetilde {   S}$
by
\begin{equation*}   
K[a, \{b_1, \dots, b_h\}] =
 [a, \{a \vee_{_{ \mathbf   S}} b_1 \vee_{_{ \mathbf   S}}
 \dots \vee_{_{ \mathbf   S}} b_h  \} ],
  \end{equation*}    
where $[a, \{b_1,  \dots, b_h\}]$ 
denotes the $ \sim$-class of the pair $(a, \{b_1,  \dots, b_h\})$.

In Lemma \ref{corre2} we shall  prove that $K$ is well-defined
and that  $\widetilde {   S}$ naturally inherits
a semilattice operation $  \vee$ from the semilattice product
$( S, \vee_{_{ \mathbf   S}}) \times \mathbf S^{{<} \omega }$.

Define $ \sqsubseteq $ on   $\widetilde {   S}$ by
 $[a, \{b_1,  \dots, b_h\}] \sqsubseteq
[c, \{d_1,  \dots, d_k\}]$ if 
 $[a, \{b_1,  \dots, b_h\}] \allowbreak \leq
K[c, \{d_1,  \dots, d_k\}]$, 
 where $\leq$ 
is the order induced by $   \vee$
and let $\widetilde {   \mathbf S} = (\widetilde { S},   \vee, \sqsubseteq ) $,
$\widetilde {   \mathbf S}' = (\widetilde { S},   \vee, K) $.

Finally, define
$\upsilon_{_\mathbf S} :  S \to \widetilde { S}$ 
by 
$\upsilon_{_\mathbf S} (a)= [a,  \emptyset ]$.
\end{definition}

\begin{remark} \labbel{*} 
We intuitively think
of  $[a, \{b_1,  \dots, b_h\}]$ as $a \vee Kb_1 \vee \dots \vee  Kb_h$,
 where $Kb_1, \dots , Kb_h$ 
are the ``new''  closures we need to introduce.
Compare Remark \ref{rough}.
In particular,  $[a, \emptyset ]$ corresponds to $a$
and $[b, \{ b \}]$ corresponds to a  new element $Kb$.  
 \end{remark}   

\begin{theorem} \labbel{univt2}
Suppose that  $\mathbf S$ is a specialization semilattice
and let $ \widetilde {  \mathbf S}$
and $\upsilon_{_\mathbf S} $
  be  as in Definition \ref{und2}.
Then the following statements hold.
  \begin{enumerate}   
 \item 
$ \widetilde {  \mathbf S}$ is
a principal 
 specialization semilattice.  
\item  
$\upsilon_{_\mathbf S} $ is an
 embedding of $\mathbf S$ 
 into $\widetilde {  \mathbf S}$.
\item  
The pair $ (\widetilde {  \mathbf S}, \upsilon_{_\mathbf S})$
has the following universal property.

For every  principal 
 specialization semilattice
$\mathbf  T$  
 and every homomorphism
 $ \eta : \mathbf S \to \mathbf  T$, there
is a unique  $K$-{\hspace{0 pt}}homomorphism  
 $ \widetilde{\eta} : \widetilde{\mathbf S} \to \mathbf  T$
such that
$\eta = \upsilon_{_\mathbf S} \circ \widetilde{\eta}$. 
\begin{equation*}
\xymatrix{
	{\mathbf S}  \ar[rd]_{\eta}
 \ar[r]^{\upsilon_{_\mathbf S}}
 &\widetilde{\mathbf  S} \ar@{-->}[d]^{\widetilde{\eta}}\\
	 &{\mathbf  T}}
   \end{equation*}    

\item
Suppose that  $\mathbf U$
is another specialization semilattice 
and $\psi : \mathbf S \to \mathbf U$ 
is a 
homomorphism. Then 
 there is a unique
 $K$-\hspace{0 pt}homomorphism
$\widetilde{\psi}: \widetilde{\mathbf S} \to \widetilde{\mathbf U}$
making the following diagram commute:

\begin{equation*}
\xymatrix{
	{\mathbf S} \ar[d]_{\psi} \ar[r]^{\upsilon_{_\mathbf S}}
 &\widetilde{\mathbf S} \ar@{-->}[d]^{\widetilde{\psi}}\\
	{\mathbf U} \ar[r]^{\upsilon_{_\mathbf U}}
 &{\widetilde{\mathbf U}}}
   \end{equation*}    
 \end{enumerate} 
 \end{theorem}

\begin{lemma} \labbel{corre2}
Assume the notation and the definitions in \ref{und2}. 
  \begin{enumerate}[(i)]    
\item
 The relation $\precsim$ from Definition \ref{und2}(a)
is reflexive and transitive on 
$ S \times  S^{{<} \omega }$, hence
 $ \sim$ from \ref{und2}(b) is an equivalence relation. 
\item
The operation $K$ is
 well-defined on the $ \sim$-equivalence classes.
\item
The relation $ \sim$ is a semilattice congruence
on the semilattice  $ ( S, \vee) \times \mathbf S^{{<} \omega }$,
hence the quotient $\widetilde{S}$
inherits  a semilattice structure from
the semilattice $( S, \vee) \times \mathbf S^{{<} \omega }$.
The join operation on the quotient is given by
\begin{equation} \labbel{3-1}
[a, \{b_1,  \dots, b_h\}]  \vee [c, \{d_1,  \dots, d_k\}] =
[a \vee _{_{ \mathbf   S}}  c ,  \{b_1,  \dots, b_h, d_1,  \dots, d_k \}], 
 \end{equation}    
and, moreover, the following holds:
\begin{equation}\labbel{30}     
\begin{aligned} ~     
[a, \{b_1,  \dots, b_h\}] & \leq [c, \{d_1,  \dots, d_k\}] 
\text{ if and only if}  
\\
(a, \{b_1,  \dots, b_h\}) & \precsim (c, \{d_1,  \dots, d_k\}).
 \end{aligned}
 \end{equation}    
 \end{enumerate} 
 \end{lemma} 

\begin{proof}  
(i) To prove that 
$\precsim$ is reflexive, notice that
if $a=c$, then condition (a1) is verified 
by an arbitrary choice of the
$d_j^*$s, e.~g., $d_j^*= d_j$, for every index $j$.
If  $ \{b_1, b_2, \dots, b_h\}= 
\{d_1, d_2, \dots, d_k\}$, then 
condition (a2) is verified because of \eqref{s4}
from Remark \ref{psc}(d). 
In the case $h=0$, that is,  
$\{b_1, b_2, \dots, b_h\} = \emptyset $,
the condition is vacuously verified by the conventions
about quantifying over the empty set.

Now we deal with transitivity of $\precsim$, so suppose that
\begin{equation*} 
(a, \{b_1, \dots, b_h\}) \precsim 
(c,\{d_1,  \dots, d_k\})\precsim 
(e,\{f_1, \dots, f_\ell\}).
 \end{equation*}     
By (a1), 
 $a \leq_{_{ \mathbf   S}} c
 \vee_{_{ \mathbf   S}} d_1^* 
\vee_{_{ \mathbf   S}} \dots 
\vee_{_{ \mathbf   S}} d_k^*$, for some
$d_1^* \sqsubseteq_{_{ \mathbf   S}} d_1$, 
\dots,   $d_k^* \sqsubseteq_{_{ \mathbf   S}} d_k$,
and 
$c \leq_{_{ \mathbf   S}} e
 \vee_{_{ \mathbf   S}} f_1^* 
\vee_{_{ \mathbf   S}} \dots 
\vee_{_{ \mathbf   S}} f_\ell^*$, for some
$f_1^* \sqsubseteq_{_{ \mathbf   S}} f_1$, 
\dots,   $f_\ell^* \sqsubseteq_{_{ \mathbf   S}} f_\ell$,
thus 
\begin{equation} \labbel{parap}  
a \leq_{_{ \mathbf   S}}  e
 \vee_{_{ \mathbf   S}} f_1^* 
\vee_{_{ \mathbf   S}} \dots 
\vee_{_{ \mathbf   S}} f_\ell^*
 \vee_{_{ \mathbf   S}} d_1^* 
\vee_{_{ \mathbf   S}} \dots 
\vee_{_{ \mathbf   S}} d_k^* .
 \end{equation}     
By (a2), for every 
$j \leq k$, there is $m \leq \ell$
such that $d_j \sqsubseteq_{_{ \mathbf   S}} f_m$.
For each $m \leq \ell$, let $D_m$ be the set of those
indices $j \leq k$ 
 such that  $d_j \sqsubseteq_{_{ \mathbf   S}} f_m$,
thus $D_1 \cup \dots \cup D_ \ell = \{ 1,2,\dots, k \} $.  
Since   $d_j^* \sqsubseteq_{_{ \mathbf   S}} d_j$, for every $j \leq k$,
then, by \eqref{s2}, we get   
$d_j^* \sqsubseteq_{_{ \mathbf   S}} f_m$, for every $j \in D_m$.
By a finite iteration of \eqref{s3},
setting $f_m^{**} =
f^*_m \vee_{_{ \mathbf   S}} \bigvee_{_{ \mathbf   S}}
 \{ \, d_j^* \mid j \leq k,  j \in D_m \, \} $, for $m \leq \ell$, 
we get $f_m^{**} \sqsubseteq_{_{ \mathbf   S}} f_m$,
since also   $f^*_m \sqsubseteq_{_{ \mathbf   S}} f_m$, by assumption
(if $D_m = \emptyset $, simply set  $f_m^{**} =
f^*_m $ and use \eqref{s4}). Notice that $\bigvee_{_{ \mathbf   S}}
 \{ \, d_j^* \mid  j \in D_m \, \}$
is a finite join.
Since $D_1 \cup \dots \cup D_ \ell = \{ 1,2,\dots, k \} $,
then 
\begin{equation} \labbel{parapip}   
 f_1^{**} 
\vee_{_{ \mathbf   S}} \dots 
\vee_{_{ \mathbf   S}} f_\ell^{**}
=
 f_1^* 
\vee_{_{ \mathbf   S}} \dots 
\vee_{_{ \mathbf   S}} f_\ell^*
 \vee_{_{ \mathbf   S}} d_1^* 
\vee_{_{ \mathbf   S}} \dots 
\vee_{_{ \mathbf   S}} d_k^* ,
 \end{equation}     
  thus
$a \leq_{_{ \mathbf   S}} e  \vee_{_{ \mathbf   S}} f_1^{**} 
\vee_{_{ \mathbf   S}} \dots 
\vee_{_{ \mathbf   S}} f_\ell^{**}$, by \eqref{parap}.      
This means that condition $(a1)$ holds, when applied to 
$(a, \{b_1, b_2, \dots, b_h\}) \precsim 
(e,\{f_1, f_2, \dots, f_\ell\})$.  
Condition (a2) holds, as well, 
since,  
for every $i \leq h$, 
there is $j \leq k$ such that   
$b_i \sqsubseteq_{_{ \mathbf   S}} d_j$
and, 
for every 
$j \leq k$, there is $m \leq \ell$
such that $d_j \sqsubseteq_{_{ \mathbf   S}} f_m$, hence we get
$b_i \sqsubseteq_{_{ \mathbf   S}} f_m$
by \eqref{s2}.

Having proved that 
$\precsim $ is reflexive and transitive, then so is
$\sim$; hence $\sim$ is an equivalence
relation, being symmetric by definition.

(ii) It is enough to show that
\begin{equation} \labbel{mum} 
\text{if  
$(a, \{b_1,  \dots, b_h\}) \precsim
(c, \{d_1,  \dots, d_k\})$, then
$ (a, \{b\} ) \precsim (c, \{ d  \} )$,}
 \end{equation}   
where we have set
$b= a \vee_{_{ \mathbf   S}} b_1 \vee_{_{ \mathbf   S}}
 \dots \vee_{_{ \mathbf   S}} b_h$ and
$d= c \vee_{_{ \mathbf   S}} d_1 \vee_{_{ \mathbf   S}} 
\dots \vee_{_{ \mathbf   S}} d_k  $. 

Indeed, from  \eqref{mum}, together with the symmetrical condition,
we get that 
\begin{multline*}
[a, \{b_1,  \dots, b_h\}] =[c, \{d_1,  \dots, d_k\}]  
\text{  implies }  
\\
K[a, \{b_1,  \dots, b_h\}] = [a, \{b \} ] =
[c, \{d \}] =
K[c, \{d_1,  \dots, d_k\}]. 
 \end{multline*}

So let us assume that 
$(a, \{b_1,  \dots, b_h\}) \precsim
(c, \{d_1,  \dots, d_k\})$.
By (a1), 
there are elements 
$d_1^* \sqsubseteq_{_{ \mathbf   S}} d_1$, 
\dots,   $d_k^* \sqsubseteq_{_{ \mathbf   S}} d_k$
such that $a \leq _{_{ \mathbf   S}} c \vee_{_{ \mathbf   S}} d^*$,
where we let $d^*= d_1^* \vee_{_{ \mathbf   S}}
 \dots \vee_{_{ \mathbf   S}} d_k^*$. 
By \eqref{s1}, for every $j \leq k$,  $d_j \sqsubseteq_{_{ \mathbf   S}} d$,
hence   $d_j^* \sqsubseteq_{_{ \mathbf   S}} d$,
by \eqref{s2}. 
Iterating \eqref{s3}, we get  
$d^* \sqsubseteq_{_{ \mathbf   S}} d$,
thus $a \leq _{_{ \mathbf   S}} c \vee_{_{ \mathbf   S}} d^*$
gives condition (a1) for the inequality
$(a, \{ b \} ) \precsim (c, \{ d \} ) $. 

Again 
since  $d_j \sqsubseteq_{_{ \mathbf   S}} d$,
for every $j \leq k$, and since, for every $i \leq h$, 
there is $j \leq k$ such that   
$b_i \sqsubseteq_{_{ \mathbf   S}} d_j$, by (a2),
then, 
\begin{equation}\labbel{pirl}    
\text{$b_i \sqsubseteq_{_{ \mathbf   S}} d$, for every $i \leq h$}
   \end{equation}
by \eqref{s2}.  
From $a \leq _{_{ \mathbf   S}} c \vee_{_{ \mathbf   S}} d^*$
and \eqref{s1} we get 
$a \sqsubseteq _{_{ \mathbf   S}} c \vee_{_{ \mathbf   S}} d^*$;
again by \eqref{s1} we have $c \sqsubseteq _{_{ \mathbf   S}} d$,
thus    $c \vee_{_{ \mathbf   S}} d^* \sqsubseteq _{_{ \mathbf   S}} d$
by \eqref{s3} and $d^* \sqsubseteq _{_{ \mathbf   S}} d$.
 Then $ a \sqsubseteq _{_{ \mathbf   S}} d$ follows  from
\eqref{s2}. Since, for every $i \leq h$, $b_i \sqsubseteq_{_{ \mathbf   S}} d$,
then $b \sqsubseteq_{_{ \mathbf   S}} d$, by iterating \eqref{s3}.
We have proved condition (a2) for   $(a, \{ b \} ) \precsim (c, \{ d \} ) $,
hence \eqref{mum} holds and this, as  we mentioned,
is enough to get (ii).

(iii) By symmetry, it is enough to show that
if  
\begin{equation}\labbel{uffa} 
   (a, \{b_1,  \dots, b_h\}) \precsim
(c, \{d_1,  \dots, d_k\}),
  \end{equation} 
  then   
\begin{equation*} 
(a, \{b_1,  \dots, b_h\})  \vee (e, \{f_1,  \dots, f_\ell\}) \precsim
(c, \{d_1,  \dots, d_k\}) \vee (e, \{f_1,  \dots, f_\ell\}),
 \end{equation*}     
  that is, 
\begin{equation}\labbel{uffi}      
(a \vee_{_{ \mathbf   S}} e, \{b_1,  \dots, b_h, f_1,  \dots, f_\ell\}) ) \precsim
(c \vee_{_{ \mathbf   S}} e, \{d_1,  \dots, d_k, f_1,  \dots, f_\ell\}).
 \end{equation}
From \eqref{uffa},
as witnessed by  $a \leq_{_{ \mathbf   S}} c
 \vee_{_{ \mathbf   S}} d_1^* 
\vee_{_{ \mathbf   S}} \dots 
\vee_{_{ \mathbf   S}} d_k^*$, for
the appropriate $d_j^*$s, we get
 $a \vee_{_{ \mathbf   S}} e
 \leq_{_{ \mathbf   S}} c \vee_{_{ \mathbf   S}} e
 \vee_{_{ \mathbf   S}} d_1^* 
\vee_{_{ \mathbf   S}} \dots 
\vee_{_{ \mathbf   S}} d_k^*
 \vee_{_{ \mathbf   S}} f_1^* 
\vee_{_{ \mathbf   S}} \dots 
\vee_{_{ \mathbf   S}} f_\ell^*$,
 by choosing, for example,  $f_m^*=f_m$, 
hence condition (a1) is verified for \eqref{uffi}.
Condition  (a2) for \eqref{uffi}
follows from \eqref{s4} and from condition (a2) for \eqref{uffa}.

To prove \eqref{30},
notice that
$[a, \{b_1,  \dots, b_h\}]  \leq [c, \{d_1,  \dots, d_k\}]$ 
 means
\begin{align}  
\nonumber 
[a, \{b_1, \allowbreak   \dots, b_h\}]  \vee  [c, \{d_1,  \dots, d_k\}] &=
[c, \{d_1,  \dots, d_k\}], \quad
\text{that is,}
\\
\nonumber
 [a \vee_{_{ \mathbf   S}} c, \{b_1,  \dots, b_h, d_1,  \dots, d_k\}] &=
[c, \{d_1,  \dots, d_k\}], \quad
\text{that is,}
\\
\labbel{parp}    (a \vee_{_{ \mathbf   S}} c, \{b_1,  \dots, b_h, d_1,  \dots, d_k\})
&\precsim
(c, \{d_1,  \dots, d_k\}),
 \end{align}
since 
$ (c, \{d_1,  \dots, d_k\})
\precsim 
(a \vee_{_{ \mathbf   S}} c, \{b_1,  \dots, b_h, d_1,  \dots, d_k\})$
is obvious.
Now, \eqref{parp} is equivalent to  
$(a, \{b_1,  \dots, b_h\})  \precsim (c, \{d_1,  \dots, d_k\})$,
since, for every element $X \in S$,
  $a \vee_{_{ \mathbf   S}} c 
\leq_{_{ \mathbf   S}} c \vee_{_{ \mathbf   S}} X$  
if and only if 
 $a   \leq_{_{ \mathbf   S}} c \vee_{_{ \mathbf   S}} X$.
Moreover, for every $j \leq k$,
$ d_j \sqsubseteq_{_{ \mathbf   S}} d_j$, by \eqref{s4}.    
\end{proof}

\begin{proof}[Proof of Theorem \ref{univt2}]
Definition \ref{und2} is justified
by Lemma \ref{corre2}.

\begin{claim} \labbel{cl2} 
$\widetilde {   \mathbf S}' = (\widetilde { S},   \vee, K) $
is  a closure semilattice.
 \end{claim}  

We have  shown in Lemma \ref{corre2}(iii)  that 
$(\widetilde { S},   \vee)$ is a semilattice; it remains to check
that $K$ is a closure operation. 
Indeed, 
\begin{align*}  
[a, \{b_1,  \dots, b_h\}]
 &\leq ^{\eqref{30}} [a, \{a \vee_{_{ \mathbf   S}} b_1 \vee_{_{ \mathbf   S}} \dots \vee_{_{ \mathbf   S}} b_h  \} ]
=K[a, \{b_1,  \dots, b_h\}]
\\
 KK[a, \{b_1,  \dots, b_h\}]&=
 K[a, \{a \vee_{_{ \mathbf   S}} b_1 \vee_{_{ \mathbf   S}} \dots \vee_{_{ \mathbf   S}} b_h  \}]
\\
&=[a, \{a \vee_{_{ \mathbf   S}} a \vee_{_{ \mathbf   S}} b_1 \vee_{_{ \mathbf   S}} \dots \vee_{_{ \mathbf   S}} b_h  \}]=K[a, \{b_1,  \dots, b_h\}] 
 \end{align*} 
 and isotony follows from \eqref{mum} and \eqref{30}.

Having proved the claim, Clause (1)
in Theorem \ref{univt2} follows 
from Remark \ref{psc}(a).
 
Now we prove (2).
We have
$\upsilon_{_\mathbf S} (a \vee_{_\mathbf S} b) =
 [a \vee_{_\mathbf S} b,  \emptyset   ]
= ^{\eqref{3-1}}  [a ,   \emptyset   ] \vee [b,  \emptyset   ]
 = \upsilon_{_\mathbf S} (a) \vee \upsilon_{_\mathbf S} ( b)$,
hence $\upsilon_{_\mathbf S} $ is a semilattice homomorphism.
Moreover, $\upsilon_{_\mathbf S} $ is injective, since
$\upsilon_{_\mathbf S} (a) =\upsilon_{_\mathbf S} ( c )$ 
means 
$(a,  \emptyset   ) \sim  (c,  \emptyset   )$
and, by Definition \ref{und2}(a1), this happens only if 
$a \leq_{_\mathbf S} c$ and $c \leq_{_\mathbf S} a$,
that is,  $a=c$.

Furthermore, 
if $ a \sqsubseteq_{_\mathbf S} d $ in $\mathbf S$,
then 
$(a ,  \emptyset ) \precsim (d ,  \{ d \} )$,
by taking $d^*=a$ in (a1),  thus
$\upsilon_{_\mathbf S} (a) = [a ,  \emptyset   ]
\leq ^{\eqref{30}} [d ,  \{ d \} ] = K[ d,  \{ \emptyset \} ]=
K\upsilon_{_\mathbf S} (d)$,
that is,  $\upsilon_{_\mathbf S} (a) \sqsubseteq 
\upsilon_{_\mathbf S} (d)$,
according to the definition  of  $ \sqsubseteq $  on
 $\widetilde{S}$  in Definition \ref{und2}.
This shows that $\upsilon_{_\mathbf S}$
is a $ \sqsubseteq $-homomorphism.

In fact, $\upsilon_{_\mathbf S}$ is an embedding, since from
$\upsilon_{_\mathbf S} (a) \sqsubseteq \upsilon_{_\mathbf S} (d)$,
that is, 
$[a, \emptyset  ] \leq  K[d, \emptyset ] =  [d, \{ d \} ] $,
we get $a \leq_{_\mathbf S} d \vee_{_\mathbf S} d^*$,
for some $d^*$ such that $d^* \sqsubseteq_{_\mathbf S} d$. Hence  
$a \sqsubseteq_{_\mathbf S}  d \vee_{_\mathbf S} d^* 
\sqsubseteq_{_\mathbf S} d$,
by \eqref{s4} and \eqref{s3},
thus $a \sqsubseteq_{_\mathbf S} d$, by \eqref{s2}.
This shows that $\upsilon_{_\mathbf S}$ is an embedding.

Having proved (2), we now deal with (3).
If $\eta: \mathbf S \to \mathbf T$
is a homomorphism and there exists 
$\widetilde{\eta}$ 
such that   
$\eta = \upsilon_{_\mathbf S} \circ \widetilde{\eta}$,
then
$\widetilde{\eta}([a, \emptyset ]) =
   \widetilde{\eta}(\upsilon_{_\mathbf S}(a))
 = \eta (a) $, for every $a \in S$.
If furthermore $ \widetilde{\eta}$
is a $K$-homomorphism, then 
$ \widetilde{\eta}([b, \{ b \}]) =   \widetilde{\eta}(K[b, \emptyset ])=
K_{_\mathbf T} \widetilde{\eta}([b, \emptyset ]) = K_{_\mathbf T} \eta (b)$.
Notice that we have assumed that
$\mathbf T$ is  principal, so that 
 $K _{_\mathbf T}\eta (b)$ does exist. 

Now notice that 
$(a, \{b_1, \dots, b_h   \}) \sim
(a \vee_{_\mathbf S}  b_1 \vee_{_\mathbf S} \dots
 \vee_{_\mathbf S}  b_h , \{b_1, \dots, b_h   \})$,
by \eqref{s4}, hence 
$[a, \{b_1, \dots, b_h   \}] =
[a \vee_{_\mathbf S}  b_1 \vee_{_\mathbf S} \dots 
\vee_{_\mathbf S}  b_h , \{b_1, \dots, b_h   \}]$.
Since $\widetilde{\eta}$ is supposed to be
a lattice homomorphism,
it follows that 
\begin{equation} \labbel{37}      
\begin{aligned} 
 &\widetilde{\eta}([a, \{b_1, \dots, b_h   \} ]) =
 \widetilde{\eta}([a \vee_{_\mathbf S}  b_1 \vee_{_\mathbf S} \dots
 \vee_{_\mathbf S}  b_h , \{b_1, \dots, b_h   \} ]) = ^{\eqref{3-1}} 
\\
& \widetilde{\eta}([a, \emptyset ] \vee [ b_1, \{b_1  \} ] \vee \dots
 \vee [ b_h , \{ b_h   \} ]) =
\\
&\widetilde{\eta}([a, \emptyset ]) \vee_{_\mathbf T}
\widetilde{\eta}([b_1 \{ b_1 \} ]) \vee_{_\mathbf T} \dots
 \vee_{_\mathbf T} 
\widetilde{\eta}([b_h \{ b_h \} ])=
\\
&\eta (a) \vee_{_\mathbf T} K_{_\mathbf T} \eta (b_1) 
\vee_{_\mathbf T} \dots \vee_{_\mathbf T} K_{_\mathbf T} \eta (b_h). 
 \end{aligned}
 \end{equation} 
Hence if $   \widetilde{\eta}$  
exists it is unique.

It is then enough to show that the above condition
$\widetilde{\eta}([a, \{b_1, \dots, b_h   \} ]) \allowbreak =
 \eta (a) \vee_{_\mathbf T} K_{_\mathbf T} \eta (b_1) \vee_{_\mathbf T}
 \dots \vee_{_\mathbf T} K_{_\mathbf T} \eta (b_h)$
actually determines a $K$-\hspace{0 pt}homomorphism
$   \widetilde{\eta}$ from 
$\widetilde{ \mathbf S}$  
to $\mathbf T$. Indeed, \eqref{37} gives 
$\widetilde{\eta}([a, \emptyset ]) = \eta (a)$,
so that we actually have
$\eta = \upsilon_{_\mathbf S} \circ \widetilde{\eta}$.  

First, we need to check that 
if  
$(a, \{b_1,  \dots, b_h\}) \sim
(c, \{d_1,  \dots, d_k\})$, then 
$\eta (a) \vee_{_\mathbf T} K_{_\mathbf T} \eta (b_1) \vee_{_\mathbf T} \dots
 \vee_{_\mathbf T} K_{_\mathbf T} \eta (b_h)
= \eta (c) \vee_{_\mathbf T} K_{_\mathbf T} \eta (d_1) \vee_{_\mathbf T}
 \dots \vee_{_\mathbf T} K_{_\mathbf T} \eta (d_k)$, so that
$   \widetilde{\eta}$ is well-defined. 

Hence assume that clauses (a1) and (a2) in Definition \ref{und2} hold.
From $d_1^* \sqsubseteq_{_\mathbf S} d_1$ 
we get $ \eta (d_1^*) \sqsubseteq_{_\mathbf T} \eta (d_1)$,
since $\eta$ is a homomorphism, hence $ K_{_\mathbf T}\eta (d_1^*)
\allowbreak  \leq_{_\mathbf T} K_{_\mathbf T}\eta (d_1)$, by Remark \ref{psc}(c), 
since $\mathbf T$ is assumed to be principal. 
Similarly for the other indices.
Hence from $a \leq_{_\mathbf S} c \vee_{_\mathbf S} d_1^*
 \vee_{_\mathbf S} \dots \vee_{_\mathbf S} d_k^*$ 
given by (a1), we get
\begin{align*} 
 \eta (a) &\leq_{_\mathbf T} \eta (c) \vee_{_\mathbf T} \eta (d_1^*)
 \vee_{_\mathbf T} \dots \vee_{_\mathbf T} \eta (d_k^*)
\\
&\leq_{_\mathbf T} 
\eta (c) \vee_{_\mathbf T} K_{_\mathbf T}\eta (d_1^*)
 \vee_{_\mathbf T} \dots \vee_{_\mathbf T} K_{_\mathbf T}\eta (d_k^*)
\\
&\leq_{_\mathbf T} \eta (c) \vee_{_\mathbf T} K_{_\mathbf T}\eta (d_1) 
\vee_{_\mathbf T} \dots \vee_{_\mathbf T} K_{_\mathbf T}\eta (d_k),
\end{align*}    
since $\eta$ is a semilattice homomorphism.
By \ref{und2}(a2),
for every $i \leq h$, there is 
 $j \leq k$ such that  
  $b_i \sqsubseteq_{_\mathbf S} d_j$. As above, 
by Remark \ref{psc}(c), we get  
$ K_{_\mathbf T}\eta (b_i) \leq_{_\mathbf T} K_{_\mathbf T}\eta (d_j)$,
hence 
$ K_{_\mathbf T}\eta (b_i) \leq_{_\mathbf T} 
 \eta (c) \vee_{_\mathbf T} K_{_\mathbf T}\eta (d_1)
 \vee_{_\mathbf T} \dots \vee_{_\mathbf T} K_{_\mathbf T}\eta (d_k)$,
for every $i \leq h$. 
Thus
\begin{align*}  
\eta (a) \vee_{_\mathbf T} K_{_\mathbf T} \eta (b_1) \vee_{_\mathbf T}
 \dots \vee_{_\mathbf T} K_{_\mathbf T} \eta (b_h)
\leq_{_\mathbf T}
\\
 \eta (c) \vee_{_\mathbf T} K_{_\mathbf T}
 \eta (d_1) \vee_{_\mathbf T} \dots \vee_{_\mathbf T} K_{_\mathbf T} \eta (d_k).
\end{align*} 
The converse inequality is proved in the symmetrical way.
We have proved  that $   \widetilde{\eta}$ is well-defined.

We now check that 
$   \widetilde{\eta}$ is a
semilattice homomorphism. 
Indeed, 
\begin{align*} 
 & \widetilde{\eta}([a, \{b_1,  \dots, b_h\}])
  \vee_{_\mathbf T} \widetilde{\eta}([c, \{d_1,  \dots, d_k\}]) =
\\
 &\eta (a) \vee_{_\mathbf T} K_{_\mathbf T}
 \eta (b_1) \vee_{_\mathbf T} \dots \vee_{_\mathbf T}
 K_{_\mathbf T} \eta (b_h) \vee_{_\mathbf T}
\eta (c) \vee_{_\mathbf T}
K_{_\mathbf T} \eta (d_1)
 \vee_{_\mathbf T} \dots \vee_{_\mathbf T} K_{_\mathbf T} \eta (d_k)=
\\
 &
\eta( a \vee_{_\mathbf S} c)  \vee_{_\mathbf T}
  K_{_\mathbf T} \eta (b_1) \vee_{_\mathbf T} \dots
 \vee_{_\mathbf T} K_{_\mathbf T} \eta (b_h) 
 \vee_{_\mathbf T}
K_{_\mathbf T} \eta (d_1)
 \vee_{_\mathbf T} \dots \vee_{_\mathbf T} K_{_\mathbf T} \eta (d_k)=
\\
 &
\widetilde{\eta}([a \vee_{_\mathbf S} c ,  \{b_1,  \dots, b_h, d_1,  \dots, d_k \}])=
^{\eqref{3-1}} 
\\ 
& \widetilde{\eta}([a, \{b_1,  \dots, b_h\}]
  \vee [c, \{d_1,  \dots, d_k\}])  ,
 \end{align*}
where we have used the definition of
$\widetilde{\eta}$, the  assumption
that $\eta$ is a semilattice homomorphism
and equation \eqref{3-1}.

Finally, $   \widetilde{\eta}$ is a
$K$-homomorphism,
since
\begin{align*} 
 &\widetilde{\eta}(K[a, \{b_1,  \dots, b_h\}])=
 \widetilde{\eta}([a, \{a \vee_{_\mathbf S} b_1 \vee_{_\mathbf S} \dots
 \vee_{_\mathbf S} b_h  \}])=
\\
&\eta(a) \vee_{_\mathbf T} 
K_{_\mathbf T}\eta(a \vee_{_\mathbf S} b_1 \vee_{_\mathbf S} 
\dots \vee_{_\mathbf S} b_h) =
\\
 & K_{_\mathbf T}\eta(a \vee_{_\mathbf S} b_1 
\vee_{_\mathbf S} \dots \vee_{_\mathbf S} b_h) =
\\
&K_{_\mathbf T}(\eta(a) \vee_{_\mathbf T} \eta(b_1)
 \vee_{_\mathbf T} \dots \vee_{_\mathbf T} \eta(b_h))
= ^{\ref{psc}} 
\\
&K_{_\mathbf T}(\eta(a) \vee_{_\mathbf T} K_{_\mathbf T}\eta(b_1)
 \vee_{_\mathbf T} \dots \vee_{_\mathbf T} K_{_\mathbf T}\eta(b_h))
=
K_{_\mathbf T} \widetilde{\eta}([a, \{b_1,  \dots, b_h\}]),
\end{align*}   
where we have used 
the definitions of $K$ and $\widetilde{\eta}$,
the assumption that $\eta$ is a homomorphism of specialization semilattices,
the fact that $K_{_\mathbf T}$ is a closure operation  
 and
Remark \ref{psc}(b). 

Clause (4) follows from Clause (3)
applied to   $\eta = \psi \circ \upsilon_{_{\mathbf U} }$.
 \end{proof}

It is necessary to ask that 
 $\widetilde{\eta}$ is 
a $K$-homomorphism in Theorem \ref{univt2}(3);
compare a parallel observation shortly before 
Remark 3.4 in \cite{mttlib}.

On the other hand, $\eta$ is not required to
preserve existing closures in Theorem \ref{univt2}(3).
In the next section we shall deal with the situation in which
$\eta$ is supposed to  preserve a specified set of 
closures in $\mathbf S$.

\section{Preserving existing closures} \labbel{pec} 

Suppose now that we want to extend 
some specialization semilattice $\mathbf S$ 
in such a way that all the  existing closures---or
 just some set of closures---are preserved. 
So, let $\mathbf S$ be a specialization semilattice 
and let $Z \subseteq S$ be a set of closures of $\mathbf S$.
We say that a homomorphism $\eta:\mathbf S \to \mathbf T$ 
\emph{preserves closures in $Z$} if, whenever
$a \in S$ and the closure $K_{_\mathbf S}a$ exists and belongs to $Z$,
then the closure of $\eta(a)$ exists in $\mathbf T$
and moreover   $\eta(K_{_\mathbf S}a)= K_{_\mathbf T}\eta(a) $.

By Remark \ref{psc}(e),  $z$ is a closure if and only if  $Kz=z$.  
In particular, if $z \in Z \subseteq S$ and $\eta$ preserves closures in
 $Z$, then $ \eta (z)=\eta(K_{_\mathbf S}z)= K_{_\mathbf T}\eta(z) $,
thus $\eta (z)$ is a closure in $\mathbf T$.

\begin{theorem} \labbel{upres}
Suppose that $\mathbf S$ is a specialization semilattice 
and $Z \subseteq S$ is a set of closures in $\mathbf S$.
Then there are a principal specialization semilattice 
$\widetilde{\mathbf S}^Z$ 
and an embedding  $\upsilon_{_\mathbf S}^Z:
\mathbf S \to \widetilde{\mathbf S}^Z$
such that $\upsilon_{_\mathbf S}^Z$
preserves  closures in $Z$
and, moreover, the universal conditions (3) - (4) in Theorem \ref{univt2}
hold  with respect to $\widetilde{\mathbf S}^Z$ 
and $\upsilon_{_\mathbf S}^Z$, limited to those homomorphisms
$\eta$ and $\psi$ which preserve closures in $Z$.  
 \end{theorem}

Before giving the proof of Theorem \ref{upres} 
we need to introduce some generalizations of  Definition \ref{und2}
and  of Lemma \ref{corre2}.

Recall the definition of the relations
$ \precsim$ and $ \sim$ from Definition \ref{und2}.
Notice that 
 $(a, \emptyset ) \precsim 
(c,\{d_1, d_2, \dots, d_k\})$ holds if and only if
clause (a1) holds in Definition \ref{und2}.

\begin{definition} \labbel{und2z}   
Suppose that  $Z \subseteq S$ is a set of closures in $\mathbf S$.
On 
$ S \times  S^{{<} \omega }$ 
define the following relations:
  \begin{enumerate}[(c)] 
   \item  
 $(a, \{b_1, b_2, \dots, b_h\}) \precsim^Z 
(c,\{d_1, d_2, \dots, d_k\})$ if
   \begin{enumerate}[({c}1)]  
  \item 
 $(a, \emptyset ) \precsim 
(c,\{d_1, d_2, \dots, d_k\})$, and
   \item 
for every $i \leq h$, either
  \begin{enumerate}[($ \alpha $)] 
   \item  
there is $j \leq k$ such that   
$b_i \sqsubseteq_{_{ \mathbf   S}} d_j$ 
(this is the same as clause (a2) in Definition \ref{und2}), or
\item[($ \beta  $)] 
there is $z \in Z$ such that 
$b_i \leq_{_{ \mathbf   S}} z$ and
  $(z, \emptyset ) \precsim 
(c,\{d_1, d_2, \dots, d_k\})$.
\end{enumerate}
 \end{enumerate}

\item[(d)]
 $(a, \{b_1, b_2, \dots, b_h\}) \sim^Z 
(c,\{d_1, d_2, \dots, d_k\})$ if both

 $(a, \{b_1, b_2, \dots, b_h\}) \precsim^Z 
(c,\{d_1, d_2, \dots, d_k\})$ and

 $(c,\{d_1, d_2, \dots, d_k\}) \precsim^Z
(a, \{b_1, b_2, \dots, b_h\}) $.
  \end{enumerate}

We shall soon prove 
that $\sim^Z$ is an equivalence relation.   
Let 
$ \widetilde {  S}^Z $
be the quotient of  
$ S \times S^{{<} \omega }$ 
under  $\sim^Z$.
Define $K^Z: \widetilde {   S}^Z \to \widetilde {   S}^Z$
by
\begin{equation*}   
K^Z[a, \{b_1, \dots, b_h\}]^Z =
 [a, \{a \vee_{_{ \mathbf   S}} b_1 \vee_{_{ \mathbf   S}}
 \dots \vee_{_{ \mathbf   S}} b_h  \} ]^Z,
  \end{equation*}    
where $[a, \{b_1,  \dots, b_h\}]^Z$ 
denotes the $ \sim^Z$-class of the pair $(a, \{b_1,  \dots, b_h\})$.    
 
Notice that when $Z= \emptyset $ we get back
the notions introduced in Definition \ref{und2}.

Notice also that the $ \sim$-class $[a, \emptyset ]$  of  $(a, \emptyset )$
is $\{ (a, \emptyset ) \}$, while   the
$ \sim^Z$-class   of  $(a, \emptyset )$
might contain more than one pair, for example, if $z \in Z$,
then $(z, \{ z \} ) \in [z, \emptyset ]^Z$.
In the course of the proof of Theorem \ref{upres} 
we shall show a slightly more general fact:
if $K_{_\mathbf S}a=z \in Z$, then
$[z, \emptyset]^Z =
[a,a]^Z$. 
\end{definition}   

\begin{lemma} \labbel{corre2Z}
The analogue of Lemma \ref{corre2} holds for
 $\precsim^Z$, $ \sim^Z$ and $ K^Z$.
 \end{lemma} 

\begin{proof} 
We shall point out the differences in comparison with 
the proof of Lemma \ref{corre2}. 

There is no change with respect to the proof
of reflexivity.
In order to continue the proof, we need some useful claims.
The next claim
 follows immediately from the definitions of
$\precsim$ and $\precsim^Z$.

\begin{claim} \labbel{diam}    $(a, \emptyset ) \precsim^Z 
(c,\{d_1,  \dots, d_k\})$ is
equivalent to 
$(a, \emptyset ) \precsim
(c,\{d_1,  \dots, d_k\})$. 
   \end{claim}  

\begin{claim} \labbel{clcl}
If $(a, \emptyset ) \precsim 
(c,\{d_1,  \dots, d_k\})\precsim^Z 
(e,\{f_1,  \dots, f_\ell\})$,
then
$(a, \emptyset ) \precsim^Z 
(e,\{f_1,  \dots, f_\ell\})$.
  \end{claim}  

The proof of Claim \ref{clcl} is similar to the proof of
transitivity of $\precsim$ in Lemma \ref{corre2},
but some additional arguments are needed.
Arguing as in  \ref{corre2},
we get equation \eqref{parap},
which we report below 
\begin{equation} \labbel{parapp}  
a \leq_{_{ \mathbf   S}}  e
 \vee_{_{ \mathbf   S}} f_1^* 
\vee_{_{ \mathbf   S}} \dots 
\vee_{_{ \mathbf   S}} f_\ell^*
 \vee_{_{ \mathbf   S}} d_1^* 
\vee_{_{ \mathbf   S}} \dots 
\vee_{_{ \mathbf   S}} d_k^* ,
 \end{equation} 
for some
$d_1^* \sqsubseteq_{_{ \mathbf   S}} d_1$, 
\dots\ 
and 
$f_1^* \sqsubseteq_{_{ \mathbf   S}} f_1$, 
\dots. 

Here it may happen that there is some $j \leq k$
such that, for no $m \leq \ell$, we have 
$ d_j \sqsubseteq_{_{ \mathbf   S}} f_m$, namely,  
we need to resort to clause $(\beta)$,
when witnessing  $(c,\{d_1,\dots, d_k\})\precsim^Z 
(e,\{f_1,  \dots, f_\ell\})$ for $j$ 
(notice: here $d_j$ is in place of $b_i$ in (c2),
since we are witnessing  $(c,\{d_1,\dots \})\precsim^Z 
(e,\{f_1,  \dots \})$). If $j$ is as above---call
 it a $(\beta)$-index---then,
by applying $(\beta)$,
we get  some $z \in Z$ such that 
$d_j \leq_{_{ \mathbf   S}} z$ and
  $(z, \emptyset ) \precsim 
(e,\{f_1,  \dots, f_\ell\})$.
By  Definition \ref{und2} (a1),
this means that 
$ z \leq_{_{ \mathbf   S}}
e \vee_{_{ \mathbf   S}}  f_1^{*j} 
\vee_{_{ \mathbf   S}} \dots \vee_{_{ \mathbf   S}} f_\ell^{*j}$, 
for some    $f_1^{*j},  \dots,  f_\ell^{*j}$ 
such that 
$f_1^{*j} \sqsubseteq _{_{ \mathbf   S}} f_1 , \dots,
f_\ell^{*j} \sqsubseteq _{_{ \mathbf   S}} f_\ell$. 
From 
$d_j^* \sqsubseteq_{_{ \mathbf   S}} d_j$
and $d_j \leq_{_{ \mathbf   S}} z$ we get
$d_j^* \sqsubseteq_{_{ \mathbf   S}} z$,
by \eqref{s1} and  \eqref{s2}.
From the implication 
(iii) $\Rightarrow $  (ii)   in Remark \ref{psc} (e)
it follows that  $d_j^* \leq_{_{ \mathbf   S}} z$,
since $z$ is a closure. Hence
\begin{equation} \labbel{spamp}
d_j^* \leq_{_{ \mathbf   S}} z \leq_{_{ \mathbf   S}}
e \vee_{_{ \mathbf   S}}  f_1^{*j} 
\vee_{_{ \mathbf   S}} \dots \vee_{_{ \mathbf   S}} f_\ell^{*j}, 
\end{equation}     

Now argue as in the proof of Lemma \ref{corre2},
this time setting   
\begin{equation} \labbel{timeset}     
 f_m^{**} =
f^*_m \vee_{_{ \mathbf   S}} \bigvee_{_{ \mathbf   S}}
 \{ \, d_j^* \mid j \leq k,  j \in D_m \, \} \vee_{_{ \mathbf   S}}
\bigvee_{_{ \mathbf   S}}
 \{ \, f_m^{*j} \mid  j \leq k, \text{ $j$  is a $(\beta)$-index} \, \} 
\end{equation} 
(or simply defining $f_m^{**}$ as in 
the proof of \ref{corre2}, in case there is no
 $(\beta)$-index).
We still have $f_m^{**} \sqsubseteq_{_{ \mathbf   S}} f_m$,
by iterating \eqref{s3}. 
On the other hand,
here $D_1 \cup \dots \cup D_ \ell $
is the set of those indices which are not $(\beta)$-indices.
However, we have
$d^*_j \leq_{_{ \mathbf   S}} 
e \vee_{_{ \mathbf   S}}  f_1^{**} 
\vee_{_{ \mathbf   S}} \dots 
\vee_{_{ \mathbf   S}} f_\ell^{**}$,
for every $(\beta)$-index $j$,   
by \eqref{spamp} and \eqref{timeset}. 
Hence
\begin{equation} \labbel{parapipp}   
 f_1^* 
\vee_{_{ \mathbf   S}} \dots 
\vee_{_{ \mathbf   S}} f_\ell^*
 \vee_{_{ \mathbf   S}} d_1^* 
\vee_{_{ \mathbf   S}} \dots 
\vee_{_{ \mathbf   S}} d_k^* 
\leq _{_{ \mathbf   S}}
e \vee_{_{ \mathbf   S}} 
f_1^{**} 
\vee_{_{ \mathbf   S}} \dots 
\vee_{_{ \mathbf   S}} f_\ell^{**},
 \end{equation}     
(compare \eqref{parapip})
  thus we still get 
$a \leq_{_{ \mathbf   S}} e  \vee_{_{ \mathbf   S}} f_1^{**} 
\vee_{_{ \mathbf   S}} \dots 
\vee_{_{ \mathbf   S}} f_\ell^{**}$, by \eqref{parapp}.      
This means that condition $(c1)$ holds, when applied to 
the relation $(a, \emptyset ) \precsim^Z 
(e,\{f_1,  \dots, f_\ell\})$.  Clause (c2)
 holds
by the conventions
about quantifying over the empty set.

Having proved Claim \ref{clcl},
we now can prove transitivity of
$\precsim^Z$. 

\emph{Transitivity of
$\precsim^Z$.} 
Assume that 
 \begin{equation} \labbel{trans}   
(a, \{b_1,  \dots, b_h\}) \precsim^Z 
(c,\{d_1,  \dots, d_k\})\precsim^Z 
(e,\{f_1,  \dots, f_\ell\}).
 \end{equation} 

By (c1) we have 
 $(a, \emptyset ) \precsim 
(c,\{d_1,  \dots, d_k\})$, 
hence 
 $(a, \emptyset ) \precsim 
(e,\{f_1,  \dots, f_\ell\})$,
by Claims \ref{diam} and  \ref{clcl}.
This shows that (c1) holds when witnessing
$(a, \{b_1, \allowbreak  \dots, \allowbreak b_h\}) \precsim^Z 
(e,\{f_1,  \dots, f_\ell\})$.  

It remains to prove (c2) witnessing 
$(a, \{b_1,  \dots, b_h \}) \precsim^Z 
(e,\{f_1,  \dots , f_\ell \})$. 
This is proved by cases, according 
to how (c2) is witnessed in both the 
relations in \eqref{trans}.

If, for some given $i \leq h$,  everything is witnessed as in ($\alpha$)
(which is the same as clause (a2) in Definition \ref{und2}),
say, $b_i \sqsubseteq_{_{ \mathbf   S}} d_j$
and $d_j \sqsubseteq_{_{ \mathbf   S}} f_m$, then
$b_i \sqsubseteq_{_{ \mathbf   S}} f_m$
by \eqref{s2}, as in the proof of Lemma \ref{corre2}(i),
hence ($\alpha$) in  (c2) is verified.

The \emph{new} cases in (c2) which are 
needed to prove $(a, \{b_1,  \dots, b_h\}) \precsim^Z 
(e,\{f_1,  \dots, f_\ell\})$
 are listed below as (*) and (**).

Case (*). For some index $i$,   (c2) for 
$(a, \{b_1,  \dots, b_h\}) \precsim^Z 
(c,\{d_1,  \dots, d_k\})$ is witnessed by
($\alpha$), say, by  $b_i \sqsubseteq_{_{ \mathbf   S}} d_j$,
while, for the index $j$,  
$(c, \{d_1,  \dots, d_k\}) \precsim^Z 
(e,\{f_1,  \dots, f_\ell\})$ 
is witnessed by ($\beta$),
thus 
$d_j \leq_{_{ \mathbf   S}} z$ and
  $(z, \emptyset ) \precsim 
(e,\{f_1,  \dots, f_\ell\})$,
for some $z \in Z$.

In this case, from 
$b_i \sqsubseteq_{_{ \mathbf   S}} d_j$
and $d_j \leq_{_{ \mathbf   S}} z$ we get
$b_i \sqsubseteq_{_{ \mathbf   S}} z$,
by \eqref{s1} and  \eqref{s2}.
From the implication 
(iii) $\Rightarrow $  (ii)   in Remark \ref{psc} (e)
it follows that  $b_i \leq_{_{ \mathbf   S}} z$,
since $z$ is a closure. Together with 
  $(z, \emptyset ) \precsim 
(e,\{f_1,  \dots, f_\ell\})$, this shows that 
clause $(\beta)$ holds for $b_i$,
when witnessing   
$(a, \{b_1,  \dots, b_h\}) \precsim^Z 
(e,\{f_1,  \dots, f_\ell\})$.

The remaining case to be treated is when:

Case (**). For some index $i$,  clause (c2) for 
$(a, \{b_1,  \dots, b_h\}) \precsim^Z 
(c,\{d_1,  \dots, \allowbreak  d_k\})$ is witnessed by
($ \beta $).

Thus 
$b_i \leq_{_{ \mathbf   S}} z$ and
  $(z, \emptyset ) \precsim 
(c,\{d_1,  \dots, d_k\})$,
for some  $z \in Z$.
From $ (c,\{d_1,  \dots, d_k\}) \precsim^Z 
(e,\{f_1,  \dots, f_\ell\})$, and Claims \ref{clcl}
and  \ref{diam}  we get 
  $(z, \emptyset ) \precsim 
(e,\{f_1,  \dots, f_\ell\})$,
thus 
$(a, \{b_1,  \dots, b_h\}) \precsim^Z 
(e,\{f_1,  \dots, f_\ell\})$ is witnessed by
($ \beta $) for $i$.

The proof of transitivity of 
$\precsim^Z$ is complete.
It follows from the properties of 
$\precsim^Z$ that $\sim^Z$
is an equivalence relation.
We now check that 

\emph{$K^Z$ is well-defined on the $\sim^Z$ equivalence classes.}  
Actually, we shall show that 
if  
$(a, \{b_1,  \dots, b_h\}) \precsim^Z
(c, \{d_1,  \dots, d_k\})$, then
$ (a, \{b\} ) \precsim (c, \{ d  \} )$
(the latter relation not $Z$-superscripted!),
where
$b= a \vee_{_{ \mathbf   S}} b_1 \vee_{_{ \mathbf   S}}
 \dots \vee_{_{ \mathbf   S}} b_h$ and
$d= c \vee_{_{ \mathbf   S}} d_1 \vee_{_{ \mathbf   S}} 
\dots \vee_{_{ \mathbf   S}} d_k  $.
This goes like the proof of Lemma \ref{corre2}(ii);
the only additional argument is given by the observation that 
if, for some index $i \leq h$,  condition 
$(\beta)$ applies, that is,  
 $b_i \leq_{_{ \mathbf   S}} z$ and
  $(z, \emptyset ) \precsim 
(c,\{d_1,  \dots, d_k\})$, thus, 
$ b_i \leq_{_{ \mathbf   S}} z \leq_{_{ \mathbf   S}}
c \vee_{_{ \mathbf   S}}  d_1^{**} 
\vee_{_{ \mathbf   S}} \dots \vee_{_{ \mathbf   S}} d_k^{**}$, 
with $d_1^{**} \sqsubseteq _{_{ \mathbf   S}} d_1 , \dots$,
then $ b_i \sqsubseteq d$, as usual, by \eqref{s1}, \eqref{s2}
and iterating \eqref{s3}.  Thus we get \eqref{pirl}
in this case, too, and, as in Lemma \ref{corre2}(ii), this is what is needed to  
complete the proof.

All the rest proceeds as in the proof of Lemma \ref{corre2}(iii). 
\end{proof}

   \begin{proof}[Proof of  Theorem   \ref{upres}]
Let $\widetilde{\mathbf S}^Z$ be the quotient
of $\widetilde{ S}^Z$   modulo 
$\sim^Z$ from Definition \ref{und2z}.
 By Lemma \ref{corre2Z}, $\widetilde{\mathbf S}^Z$
has a semilattice structure. The analogue of Claim \ref{cl2}
shows that $K^Z$ is a closure operation on $\widetilde{ S}^Z$,
thus $\widetilde{\mathbf S}^Z$ is a specialization semilattice 
with the usual definition 
$ [a,\{b_1, \dots\}]^Z \sqsubseteq^Z [c,\{d_1, \dots\}]^Z$ if
 $[a,\{b_1, \dots\}]^Z \leq^Z K^Z [c,\{d_1, \dots\}]^Z$. 

Define $\upsilon_{_\mathbf S}^Z:
S \to \widetilde{ S}^Z$ by
$\upsilon_{_\mathbf S}^Z(a) = [a, \emptyset ]^Z$.
As in the proof of Theorem \ref{univt2}, 
$\upsilon_{_\mathbf S}^Z$
is an embedding from
$ \mathbf S $ to $ \widetilde{\mathbf S}^Z$.

If $K_{_\mathbf S}a=z \in Z$, then
$(z, \emptyset) \precsim (a,a)$, since
$ z \sqsubseteq a$ (here we take $d_1=a$
and $d_1^*=z$ in Definition \ref{und2}(a1)). Moreover,
$ (a,a) \precsim^Z (z, \emptyset)$,
since $a \leq _{_\mathbf S} K_{_\mathbf S}a =z$,
and applying ($\beta$) in Definition \ref{und2z} with $a=b_1$.   
Henceforth
$[z, \emptyset]^Z =
[a,a]^Z$, thus 
$\upsilon_{_\mathbf S}^Z(K_{_\mathbf S}a)=
\upsilon_{_\mathbf S}^Z(z)=[z, \emptyset]^Z =
[a,a]^Z=K[a, \emptyset]^Z=K\upsilon_{_\mathbf S}^Z(a)$
This shows that 
$\upsilon_{_\mathbf S}^Z$ preserves closures in $Z$.  

Now we prove the $^Z$-analogue of 
clause (3) in Theorem \ref{univt2}.  
So, suppose that $\mathbf T$ is a 
principal specialization semilattice   
and $\eta: \mathbf S \to \mathbf T$  
is a homomorphism which preserves closures in $Z$. 
We could repeat the arguments in the proof of 
Theorem \ref{univt2}, but it is simpler to apply
the theorem itself.
In particular, Theorem \ref{univt2} provides
a homomorphism
$ \widetilde{\eta} : \widetilde{\mathbf S} \to \mathbf  T$
such that
$\eta = \upsilon_{_\mathbf S} \circ \widetilde{\eta}$. 
We can also apply Theorem \ref{univt2}
with 
 $ \widetilde{\mathbf S}^Z$ and
$\upsilon_{_\mathbf S}^Z$
in place of 
$\mathbf T$ and $\eta$, getting   
a surjective  homomorphism 
$\widetilde{\upsilon} :\widetilde{\mathbf S} \to \widetilde{\mathbf S}^Z$ 
such that 
$\upsilon_{_\mathbf S}^Z =\upsilon_{_\mathbf S} \circ \widetilde{\upsilon}$.
We shall show that $\widetilde{\eta}$ ``passes to the quotient''
inducing a homomorphism  
$\widetilde{\eta}^Z: \widetilde{\mathbf S}^Z \to \mathbf T$
making the following diagram commute:

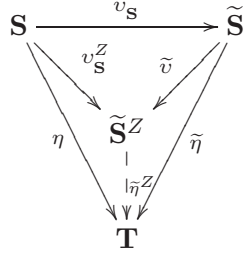
\begin{equation*}
\xymatrix{
	{\mathbf S}  \ar[rdd]_{\eta}\ar[rd]^{\upsilon_{_\mathbf S}^Z}
 \ar[rr]^{\upsilon_{_\mathbf S}}
 &&\widetilde{\mathbf S}\ar[ld]_{\widetilde \upsilon}\ar[ldd]^{\widetilde{\eta}}
\\
&\widetilde{\mathbf  S}^Z \ar@{-->}[d]^{\hspace{-2.5pt}^{\widetilde{\eta}^Z}}
\\
	 &{\mathbf  T}}
   \end{equation*}    

Since $\widetilde{\upsilon}$ sends
$[a, \{b_1,  \dots, b_h\}]$ to
$\upsilon_{_\mathbf S}^Z (a, \{b_1,  \dots, b_h\})=
[a, \{b_1,  \dots, b_h\}]^Z$, 
 it is enough to check that if
$(a, \{b_1,  \dots, b_h\}) \sim^Z 
(c,\{d_1,  \dots,   d_k\})$,
then
$\widetilde{ \eta} (a, \{b_1,  \dots, b_h\}) =
\widetilde{\eta} (c,\{d_1,  \dots,   d_k\})$,
that is, 
$\eta (a) \vee_{_\mathbf T} K_{_\mathbf T} \eta (b_1) \vee_{_\mathbf T} \dots
 \vee_{_\mathbf T} K_{_\mathbf T} \eta (b_h)
= \eta (c) \vee_{_\mathbf T} K_{_\mathbf T} \eta (d_1) \vee_{_\mathbf T}
 \dots \vee_{_\mathbf T} K_{_\mathbf T} \eta (d_k)$,
recalling the definition of 
$\widetilde{\eta}$ from the proof of Theorem \ref{univt2}. 
Following the proof of Theorem \ref{univt2},
it is enough to show that 
\begin{equation}\labbel{pappi}    
K_{_\mathbf T} \eta (b_j) \leq_{_\mathbf T}
\eta (c) \vee_{_\mathbf T} K_{_\mathbf T} \eta (d_1) \vee_{_\mathbf T}
 \dots \vee_{_\mathbf T} K_{_\mathbf T} \eta (d_k),
   \end{equation}
when $j$ is a ($\beta$)-index, when witnessing
$(a, \{b_1,  \dots, b_h\}) \precsim^Z 
(c,\{d_1,  \dots,   d_k\})$. 
Recall the definition of  a ($\beta$)-index
from the proof of Claim \ref{clcl}. 
This means that $b_j \leq _{_\mathbf S} z$,
for some $z \in Z$, thus  
$ \eta (b_j) \leq _{_\mathbf T} \eta (z)$,
since $\eta$ is a homomorphism, hence
$K_{_\mathbf T} \eta (b_j) \leq _{_\mathbf T} K_{_\mathbf T}\eta (z)
= \eta (K_{_\mathbf S} z) = \eta (z)$,
since $z$ is a closure and, by assumption,  $\eta$ preserves
closures in $Z$. Since $j$ is a ($\beta$)-index,
we also have 
$(z, \emptyset ) \precsim 
(c,\{d_1,  \dots,   d_k\})$
and this implies  
$ \eta(z) \leq_{_\mathbf T}
\eta (c) \vee_{_\mathbf T} K_{_\mathbf T} \eta (d_1) \vee_{_\mathbf T}
 \dots \vee_{_\mathbf T} K_{_\mathbf T} \eta (d_k)$,
by the proof of Theorem \ref{univt2},
thus \eqref{pappi} follows.  

Everything else goes as in the proof of Theorem \ref{univt2}.
\end{proof}


\begin{thebibliography}{99}    



\bibitem{Bl} Blass, A.,
 \emph{Combinatorial cardinal characteristics of the continuum},
 in Foreman, M., Kanamori, A. (eds.),
``Handbook of set theory'',
 Springer, Dordrecht, 
 395--489  (2010).


\bibitem{E} Ern\'{e}, M.,
 \emph{Closure},
in Mynard, F.,  Pearl E. (eds), 
``Beyond topology'',
  Contemp. Math.
\textbf{486},
Amer. Math. Soc., Providence, RI,
163--238
(2009).


\bibitem{GT} Galatos, N., Tsinakis, C.,
\emph{Equivalence of consequence relations: an order-theoretic and
  categorical perspective},
{J. Symbolic Logic}
\textbf{74},
780--810 (2009).


\bibitem{H} 
 Hodges, W.,
\emph{Model theory},
Encyclopedia of Mathematics and its Applications
\textbf{42},
Cambridge University Press, Cambridge
(1993).

\bibitem{KP} Kronheimer, E. H., Penrose, R.,
\emph{On the structure of causal spaces},
Proc. Cambridge Philos. Soc.
\textbf{63},
481--501 (1967).


\bibitem{L} Lehrer,  E.,
   \emph{On a representation of a relation by a measure},
   J. Math. Econom. \textbf{20}, 107--118  (1991).

\bibitem{mtt} Lipparini, P., \emph{A model theory of topology},
arXiv:2201.00335v1, 1--30 (2022).

\bibitem{mttlib}  Lipparini, P., 
\emph{Universal extensions of specialization semilattices}, 
Categ. Gen. Algebr. Struct. Appl. \textbf{17}, 
101--116  (2022).

\bibitem{sapimpap} 
Lipparini, P., 
\emph{Preservation of superamalgamation by expansions}, arXiv:2203.10570, 1--12 (2022),  extended version to be posted.




\end{thebibliography}
\end{document}